\documentclass[preprint,11pt]{elsarticle}
\biboptions{numbers,sort&compress}
\usepackage{amsfonts,psfrag,amsmath,mathrsfs,amsthm,color,empheq}
\usepackage{color,amssymb}
\usepackage{epsfig}
\usepackage[hidelinks]{hyperref}
\usepackage{multirow}

 \renewenvironment{thebibliography}[1]{%
 	\begin{oldthebibliography}{#1}%
 		\setlength{\parskip}{0.1ex}%
 		\setlength{\itemsep}{0.1ex}%
 	}%
 	{%
 	\end{oldthebibliography}%
 }

\usepackage{amsmath,mathtools} % assumes amsmath package installed
\usepackage{amsthm,amssymb,bbm}
 \usepackage{graphicx,float}
 \usepackage{setspace}

\usepackage{mathptmx}
\usepackage[T1]{fontenc}
 \usepackage[margin=1in]{geometry}

\newcommand{\tn}[1]{\quad \textnormal{#1}\quad }

\newcommand{\IR}{{\mathbb{R}}}
\newcommand{\IC}{{\mathbb{C}}}

\newcommand{\IF}{{\mathbb{F}}}
\newcommand{\IA}{{\mathbb{A}}}
\newcommand{\Ii}{{\mathrm{i}}}

\newcommand{\CO}{{\mathcal{O}}}

\newcommand{\T}{{\intercal}}
\newcommand{\diag}{{\rm{diag}}}

\newcommand{\bmt}{\left[ \begin{array}{ccccccccc}}
\newcommand{\emt}{\end{array}\right]}
\newcommand{\bean}{\begin{eqnarray*}}
\newcommand{\eean}{\end{eqnarray*}}
\newcommand{\bea}{\begin{eqnarray}}
\newcommand{\eea}{\end{eqnarray}}
\newcommand{\eq}{\begin{equation}\begin{array}{lllllllll}}
\newcommand{\ee}{\end{array}\end{equation}}
\newcommand{\eqn}{\begin{equation*}\begin{array}{lllllllll}}
\newcommand{\een}{\end{array}\end{equation*}}
\newtheorem{theorem}{Theorem}[section]

\allowdisplaybreaks

\graphicspath{{./Figures/}}
\newcommand{\blue}[1]{{\color{blue}{#1}}}
\newcommand{\red}[1]{{\color{red}{#1}}}
\makeatletter
\def\ps@pprintTitle{%
	\let\@oddhead\@empty
	\let\@evenhead\@empty
	\def\@oddfoot{\centerline{\thepage}}%
	\let\@evenfoot\@oddfoot}
\makeatother
\begin{document}
 
	\begin{frontmatter}
	 
	\title{Fast Parallel-in-Time Quasi-Boundary Value Methods for\\ Backward Heat Conduction Problems}
	 
	\author{Jun Liu}
	\ead{juliu@siue.edu}  
	 
	\address{Department of Mathematics and Statistics, Southern Illinois University Edwardsville, Edwardsville, IL 62026, USA.}

	\begin{abstract}
	 In this paper we proposed two new quasi-boundary value methods for regularizing the ill-posed backward heat conduction problems. With a standard finite difference discretization in space and time, the obtained all-at-once nonsymmetric sparse linear systems have the desired block $\omega$-circulant structure, which can be utilized to design an efficient parallel-in-time (PinT) direct solver that built upon an explicit FFT-based diagonalization of the time discretization matrix. 
	 Convergence analysis is presented to justify the optimal choice of the regularization parameter.
	 Numerical examples are reported to validate our analysis and illustrate the superior computational efficiency of our proposed PinT methods.
	  
	\end{abstract}
	 
	\begin{keyword} 
	ill-posed \sep quasi-boundary value method\sep regularization \sep	$\omega$-circulant\sep diagonalization \sep parallel-in-time   
	\end{keyword}
	
\end{frontmatter}

%%%%%%%%%%%%%%%%%%%%%%%%%%%%%%%%%%%%%%%%%%%%%%%%%%%%%%%%%%%%%%%%%%%%%%%%%%%%%%%%
\section{Introduction}
\label{SecProblem}
Let $T>0$ and $\Omega$ be an open and bounded domain in $\IR^d  (d=1,2,3)$ with a piecewise smooth boundary $\partial\Omega$.
We consider the classical backward heat conduction problem (BHCP) of reconstructing the initial data $z(\cdot,0)\in H_0^1(\Omega)$ from the final time condition $g=z(\cdot,T)  \in H_0^1(\Omega)$ according to a  homogeneous heat equation
\eq \label{state}
\left\{\begin{array}{ll}
	 z_t -\Delta z =0,\ \qquad &\tn{in} \Omega\times (0,T),  \qquad
	 z=0, \tn{on} \partial\Omega\times (0,T), \\ 
	z(\cdot,T)=g , &\tn{in} \Omega,
\end{array}\right.
\ee
This gives a  severely ill-posed linear inverse problem  that requires effective regularization techniques for stable and accurate numerical approximations \cite{engl2000,kabanikhin2011inverse}.
Our proposed methods also work for other more general boundary conditions and spatial differential operators. 
In practice, the exact final condition $g$ is always unknown and we only have a noisy measurement $g_\delta\in H_0^1(\Omega)$, which is assumed to  satisfy $\|g-g_\delta\|_{2}\le \delta$ with an estimated noise level $\delta>0$.  
 
Due to the wide applications of BHCP, there appeared many different  regularization methods in literature, such as
quasi-boundary value methods \cite{Chiwiacowsky2003,Denche_2005,quan2009new,Dinh2012},
optimal filtering method \cite{Seidman_1996,Ternat_2012},
kernel--based method \cite{Ames_1997,Hon_2010,van2017posteriori},
Tikhonov regularization method \cite{Tautenhahn_1996,Zhao_2011,cheng2014regularization,Duda_2017,cheng2020backward},
optimization method \cite{kabanikhin2009convergence,munch2017inverse},
optimal control method \cite{liu2019quasi,langer2021space},
total--variation regularization method \cite{Wang_2013},
(Fourier) truncation regularization method \cite{Nam_2010,minh2018two},  
fundamental solution method \cite{Cheng_2008},
Lie-group shooting method \cite{Chen_2019},
meshless method \cite{Ku_2019}, 
and
homotopy analysis method \cite{Liu_2018}.
The majority of these works focuses on discussing the approximation accuracy and convergence rates under various different assumptions,  while the computational efficiency of each method is inadequately investigated although it does often vary greatly among different approaches.
Fast solvers for the resultant discretized regularized linear systems were rarely discussed.

The simple quasi-boundary value method (QBVM) was developed in \cite{Clark94}, which
improves the earlier quasi-reversibility regularization methods \cite{Lions1969,Muzylev_1977}.
Based on QBVM, a modified quasi-boundary value method (MQBVM) was proposed in \cite{Denche_2005}, which shows a better convergence rate than QBVM under stronger assumptions.
Both QBVM and MQBVM regularize the ill-posed backward initial value problem
through approximating it by a well-posed boundary value problem that
depends on a regularization parameter $\alpha>0$.
For the general non-homogeneous problems, several different extensions based on truncated series 
were discussed in \cite{Nam_2010,Tuan_2015}.
In our recent work \cite{liu2019quasi}, we proposed a virtual optimal control formulation which can unify both QBVM and MQBVM by choosing some special Tikhonov regularization terms. 
Nevertheless, fast solvers for the nonsymmetric linear systems arising from QBVM and MQBVM were not developed yet in the literature, which are crucial to their applications to large-scale problems.

Meanwhile, with the advent of massively parallel computers, many efficient parallelizable numerical algorithms  for solving evolutionary PDEs have been developed in the last few decades. Besides the achieved high parallelism in space, we have seen a lot of recent advances in various parallel-in-time (PinT) algorithms \footnote[1]{We refer  to the website \url{http://parallel-in-time.org} for a comprehensive list of references on various PinT algorithms.} for solving forward time-dependent PDE problems \cite{gander201550}. 
However, the application of such PinT algorithms to ill-posed backward heat conduction problems were rarely investigated in the literature, except in one short paper \cite{daoud2007stability} about the \textit{parareal} algorithm for a different parabolic inverse problem and another earlier paper \cite{lee2006parallel} based on numerical (inverse) Laplace transform techniques in time. 
One obvious difficulty is how to address the underlying regularization treatment in the framework of PinT algorithms, which seems to be highly dependent on the problem structure.
Inspired by several recent works \cite{MR08,MPW18,GH19,LW20} on diagonalization-based PinT algorithms,  we propose to redesign the existing quasi-boundary value methods in a structured way such that the diagonalization-based PinT direct solver can be directly employed, which can greatly speed up the quasi-boundary value methods without degrading their approximation accuracy.

The paper is organized as follows.
In the next section we review two quasi-boundary value methods
based on a standard finite difference scheme in space and time
In Section 3, two new quasi-boundary value methods are proposed,
where the derived block $\omega$-circulant structured linear systems are then solved by a diagonalization-based PinT direct solver. Convergence analysis with the optimal choice of regularization parameter is discussed in Section 4.
Two numerical examples are presented  in Section 5 to demonstrate the promising efficiency of our proposed method 
and some conclusions are given in Section 6.
\section{Two quasi-boundary value methods and their finite difference scheme}
We now give a brief review of the quasi-boundary value methods, which were widely used in inverse due to its simplicity and effectiveness.
We consider a 2D square domain $\Omega=(0,L)^2$ with finite difference discretization in space and time, which can be easily adapted for 1D and 3D regular space domains.
With suitable modification,
the finite element discretization can also be used within our propose method for general spatial domains.
We partition the time interval $[0,T]$ uniformly into
$0=t_0<t_1<\cdots<t_N=T$  with $t_k-t_{k-1}=\tau=T/N$,
and discretize the space domain $\Omega=(0,L)^2$ uniformly into a uniform mesh $(\xi_i,\zeta_j)$ with
$0=\xi_0<\xi_1<\cdots<\xi_{M}=L$
and $0=\zeta_0<\zeta_1<\cdots<\zeta_{M}=L$,
 $h=\xi_i-\xi_{i-1}=\zeta_j-\zeta_{j-1}=L/M$.
For any function $y$, define the concatenated vector $y^n=(y^n_{ij})_{i=1,j=1}^{M-1,M-1}$ with 
$y^n_{ij}$  being the finite difference approximation
of $y(\xi_i,\zeta_j,t_n)$ over all interior nodes at time $t_n$.
Let $\Delta_h$ denotes the discretized Laplacian  matrix with the five-point center finite difference and the homogeneous Dirichlet boundary conditions being enforced. Let $N_t=N+1$ and  $N_x=(M-1)^2$.
Let $I_x\in\IR^{N_x\times N_x}$ and $I_t\in\IR^{N_t\times N_t}$ be  identity matrices.
 
\subsection{The Quasi-Boundary Value Method (QBVM)}
Applying the QBVM \cite{Clark94} to BHCP (\ref{state}),
 one needs to solve the following quasi-boundary value problem  
\eq \label{stateQBVPDelta}
\left\{\begin{array}{l}
 y_t  -\Delta y =0\ \tn{in} \Omega\times (0,T), \qquad
y=0 \tn{on} \partial\Omega\times (0,T) ,\\
y  (\cdot,T)+\blue{\alpha y  (\cdot,0)}=g \tn{in}\ \Omega ,
\end{array}\right.
\ee
where the regularization term $\blue{\alpha y  (\cdot,0)}$ with a parameter $\alpha>0$ was introduced into the final condition.
In noise-free setting the regularized solution $y$ converges to the original exact solution $z$ uniformly as $\alpha$ goes to zero \cite{Clark94}.
With the backward Euler scheme  in time and center finite difference in space, we obtain the full discretization scheme
\begin{align}
\blue{\alpha y^0}+ y^N &=g,   \\
\frac{y^{n}-y^{n-1}}{\tau}-\Delta_h y^n&=0
, \quad n=1,2,\cdots,N
\end{align}
which can be written into a nonsymmetric sparse linear system 
\begin{align} \label{linsysQBVP}
   A_h y_h=  f_h,
\end{align}
where 
\begin{align*}
  A_h&=\bmt
\blue{\alpha I_x}& 0&0 &\cdots &0 &I_x\\ 
-{I_x}/{\tau} &  -\Delta_h+{I_x}/{\tau} &0 & \cdots &0 &0\\
0&-{I_x}/{\tau} &  -\Delta_h+{I_x}/{\tau} & 0 & \cdots &0\\
0&0&\ddots &\ddots &\ddots &0\\
0& 0&\cdots &-{I_x}/{\tau} & -\Delta_h+{I_x}/{\tau} & 0\\
0&0&\cdots & 0 & {-I_x}/{\tau} &-\Delta_h+{I_x}/{\tau}
\emt, 
y_h=\bmt y^0\\ y^1 \\y^2 \\ \vdots \\ y^{N-1}\\ y^N \emt,
  f_h=\bmt g\\ 0 \\0 \\ \vdots \\ 0\\0 \emt.
\end{align*}
Here the $(1,1)$ block $\blue{\alpha I_x}$ of $A_h$ is from the regularization term $\blue{\alpha y  (\cdot,0)}$. Notice that  $A_h$ has a block Toeplitz or circulant structure except the first block row.
Such a block Toeplitz or circulant structure is highly desirable for the development of fast system solvers.
The key idea of our proposed PinT-QBVM is to redesign the regularization term such that 
a block $\omega$-circulant structure is achieved, which can then be solved by a fast diagonalization-based PinT direct solver.
\subsection{The Modified Quasi-Boundary Value Method (MQBVM)}
Similarly, applying the MQBVM \cite{Denche_2005} to our BHCP (\ref{state}),
one needs to solve the quasi-boundary value problem  
\eq \label{stateQBVPDeltaModified}
\left\{\begin{array}{l}
	y_t  -\Delta y =0\ \tn{in}  \Omega\times (0,T), \qquad
	y   =0 \tn{on}  \partial\Omega\times (0,T) ,\\
	y  (\cdot,T)-\blue{\alpha y_t (\cdot,0)}=g \tn{in}\ \Omega .
\end{array}\right.
\ee
where a different regularization term $\blue{\alpha y_t (\cdot,0)}$ was used.
After applying the backward Euler scheme in time and center difference in space, we obtain
the full discretization scheme
\begin{align}
	-\blue{ \alpha\frac{y^1-y^0}{\tau}}+ y^N  &=g ,   \\
	\frac{y^{n}-y^{n-1}}{\tau}-\Delta_h y^n&=0
	, \quad n=1,2,\cdots,N
\end{align}
which can be further reformulated into a nonsymmetric sparse linear system 
\begin{align} \label{linsysQBVPmodified}
	B_h y_h=  f_h,
\end{align}
where 
\begin{align}
	B_h&=\bmt
	\blue{\alpha {I_x}/{\tau}}&\blue{-\alpha {I_x}/{\tau}}&0 &\cdots &0 & I_x  \\ 
	-{I_x}/{\tau} &  -\Delta_h+{I_x}/{\tau} &0 & \cdots &0 &0\\
	0&-{I_x}/{\tau} &  -\Delta_h+{I_x}/{\tau} & 0 & \cdots &0\\
	0&0&\ddots &\ddots &\ddots &0\\
	0& 0&\cdots &-{I_x}/{\tau} & -\Delta_h+{I_x}/{\tau} & 0\\
	0&0&\cdots & 0 & {-I_x}/{\tau} &-\Delta_h+{I_x}/{\tau}
	\emt,\quad
	f_h=\bmt g\\ 0 \\0 \\ \vdots \\ 0\\0 \emt.
\end{align}
Here $B_h$ has the same entries as $A_b$ except the first block row.
Again, the nonzero $(1,2)$ block $\blue{-\alpha {I_x}/{\tau}}$ of $B_h$ seems to more destructive to the anticipated block Toeplitz/circulant structure. In the next section, we will show this annoying $(1,2)$ block can be removed by manipulating the PDE itself under a reasonable regularity assumption.

The nonsymmetric sparse linear systems (\ref{linsysQBVP}) and (\ref{linsysQBVPmodified}) are of dimension $N_tN_x$, which can be very costly to solve. The coupling of all time steps further increases the computational burden since we have to solve all time steps simultaneously in one-shot.
A general sparse direct solver often has a complexity of order $\mathcal{O}(N_t^3N_x^3)$, which is prohibitive for 2D/3D problems with a fine mesh.  In such large-scale cases, iterative solvers (e.g. Krylov subspace methods) are often the preferred choice, but effective and efficient preconditioners are required to achieve reasonably fast convergence rates. Moreover, for indefinite nonsymmetric systems, the rigorous convergence analysis of preconditioned iterative methods is a daunting task. In this paper, as the first step to apply PinT algorithms to ill-posed BHCP, we will concentrate on developing efficient diagonalization-based PinT direct solvers, which however can always be used as effective preconditioners within the framework of preconditioned iterative methods, especially for nonlinear problems.
Fortunately, regularization is a versatile concept that can be exploited for generating better structured systems that are suitable for fast solvers. 
This simple yet powerful philosophy gains little attention in literature of BHCP.
\section{Two new quasi-boundary value methods and their PinT implementation}
Notice the matrices $A_h$ and $B_h$ from QBVM and MQBVM have a block Toeplitz structure except the first row due to the chosen regularization term. 
In this section, we redesign both QBVM and MQBVM to achieve a block $\omega$-circulant structure of the system matrices, which consequently leads to a diagonalization-based PinT direct solver.
We strive to retain their good regularization effects by keeping the original regularization terms of QBVM and MQBVM. At the same time, additional new terms are introduced to the final condition equation to twist the system structure.
\subsection{A PinT quasi-boundary value method (PinT-QBVM)}
To obtain the block $\omega$-circulant structure that suitable for the diagonalization-based PinT algorithm, we propose to add an extra regularization term in QBVM as the following (compare with  (\ref{stateQBVPDelta}))
\eq \label{stateQBVPDeltaPinT}
\left\{\begin{array}{l}
	y_t  -\Delta y =0\ \tn{in}  \Omega\times (0,T), \qquad
	y   =0 \tn{on}  \partial\Omega \times (0,T),\\
	y  (\cdot,T)+\alpha(y  (\cdot,0)-\red{\tau\Delta y  (\cdot,0)})=g \tn{in}\ \Omega ,
\end{array}\right.
\ee
where $\tau=T/N$ is the same time step size $\tau=T/N$ to be used in its time discretization.
In certain sense, we can review $\tau$ as an additional mesh-dependent regularization parameter,
which would go to zero as the mesh is refined. Hence, we indeed only introduced a small perturbation term to the original regularization term of QBVM.
In numerical simulations, we find that adding this extra regularization term indeed also leads to more  accurate reconstruction.

Applying the backward Euler scheme in time
and center finite difference scheme in space, upon dividing both sides of the final condition equation by $(\tau\alpha)$ we obtain the full scheme
\begin{align}
	y^0/\tau-\Delta_h y^0+ y^N/(\tau\alpha) &=g/(\tau\alpha),   \\
	\frac{y^{n}-y^{n-1}}{\tau}-\Delta_h y^n&=0 
	, \quad n=1,2,\cdots,N
\end{align}
which can be further reformulated into a nonsymmetric sparse linear system
\begin{align} \label{linsysQBVPPinT}
	\widehat A_h y_h= \widehat f_h,
\end{align}
where (with $\omega=-1/\alpha$)
\begin{align*}
	\widehat A_h&=\bmt
	-\Delta_h+{I_x}/{\tau}& 0&0 &\cdots &0 &-\omega I_x/{\tau}\\ 
	-{I_x}/{\tau} &  -\Delta_h+{I_x}/{\tau} &0 & \cdots &0 &0\\
	0&-{I_x}/{\tau} &  -\Delta_h+{I_x}/{\tau} & 0 & \cdots &0\\
	0&0&\ddots &\ddots &\ddots &0\\
	0& 0&\cdots &-{I_x}/{\tau} & -\Delta_h+{I_x}/{\tau} & 0\\
	0&0&\cdots & 0 & {-I_x}/{\tau} &-\Delta_h+{I_x}/{\tau}
	\emt,  
	& 
	\widehat f_h=\bmt g/(\tau\alpha)\\ 0 \\0 \\ \vdots \\ 0\\0 \emt.
\end{align*}
Here the matrix $\widehat A_h$ has the desired block $\omega$-circulant structure, which is crucial to our PinT direct solver.

Assuming $y$ is differentiable in time at $t=0$ such that $y_t (\cdot,0)=\Delta y(\cdot,0)$, then we can rewrite the final condition in (\ref{stateQBVPDeltaPinT}) into
\[
y  (\cdot,T)+\alpha y(\cdot,0)-\alpha \tau y_t (\cdot,0) =g  \tn{in}\ \Omega,
\]
which can be interpreted as a weighted combination of both QBVM and MQBVM. This mathematical reformulation itself does not lead to a block $\omega$-circulant structure, but it explains why PinT-QBVM may work better than QBVM.
\subsection{A PinT modified quasi-boundary value method (PinT-MQBVM)}
Similarly, we can redesign the MQBVM \cite{Denche_2005} 
by adding an additional regularization term (compared with (\ref{stateQBVPDeltaModified}))
\eq \label{stateQBVPDeltaModifiedPinT}
\left\{\begin{array}{l}
	y_t  -\Delta y =0\ \tn{in}  \Omega\times (0,T), \qquad
	y  (\cdot,t)=0 \tn{on}  \partial\Omega\times (0,T) ,\\
	y  (\cdot,T)-\alpha (y_t (\cdot,0)-\red{y(\cdot,0)/\tau})=g \tn{in}\ \Omega .
\end{array}\right.
\ee
Clearly,  as the mesh step size $\tau$ gets smaller, we need to choose $\alpha$ such that the new term $\alpha\red{y(\cdot,0)/\tau}$ is well controlled. 

Again, assuming $y$ is differentiable at $t=0$ such that $y_t (\cdot,0)=\Delta y(\cdot,0)$, upon dividing both sides by $\alpha$ we can rewrite the final condition  equation in (\ref{stateQBVPDeltaModifiedPinT}) into
\[
y  (\cdot,T)/\alpha-(\Delta y(\cdot,0)-y(\cdot,0)/\tau)=g/\alpha  \tn{in}\ \Omega .
\]
Applying the backward Euler scheme in time and center finite difference scheme in space, we obtain the full scheme
\begin{align}
	-   \Delta_h y^0 +y^0/\tau+ y^N/\alpha &=g/\alpha ,   \\
	\frac{y^{n}-y^{n-1}}{\tau}-\Delta_h y^n&=0 
	, \quad n=1,2,\cdots,N
\end{align}
which can be further reformulated into a nonsymmetric sparse linear system
\begin{align} \label{linsysQBVPmodifiedPinT}
	\widehat B_h y_h=\widetilde f_h,
\end{align}
where (with $\omega=-{\tau}/{\alpha}$)
\begin{align*}
	\widehat B_h&=\bmt
	-\Delta_h+{I_x}/{\tau}& 0&0 &\cdots &0 &-\omega I_x/\tau \\ 
	-{I_x}/{\tau} &  -\Delta_h+{I_x}/{\tau} &0 & \cdots &0 &0\\
	0&-{I_x}/{\tau} &  -\Delta_h+{I_x}/{\tau} & 0 & \cdots &0\\
	0&0&\ddots &\ddots &\ddots &0\\
	0& 0&\cdots &-{I_x}/{\tau} & -\Delta_h+{I_x}/{\tau} & 0\\
	0&0&\cdots & 0 & {-I_x}/{\tau} &-\Delta_h+{I_x}/{\tau}
	\emt,  
	& 
	\widetilde f_h=\bmt g/\alpha\\ 0 \\0 \\ \vdots \\ 0\\0 \emt.
\end{align*}
Notice $\widehat B_h$ is the exactly same as $\widehat A_h$ except with a different $\omega$ value, and the first block of $\widetilde f_h$ and $\widehat f_h$ is also different.
 
\subsection{A diagonalization-based PinT direct solver}
In contrast to $A_h$ and $B_h$, both $\widehat A_h$ and $\widehat B_h$ have the same block $\omega$-circulant structure, which will be utilized to design a PinT direct solver as explained below (based on $\widehat A_h$). For concise description, we will use the reshaping operations \cite[p. 28]{golub2012matrix}: matrix-to-vector $\texttt{vec}$ and vector-to-matrix $\texttt{mat}$. 
With the Kronecker product notation, we can write 
\begin{align}
	\widehat A_h=\frac{1}{\tau}C_\omega\otimes I_x - I_t\otimes \Delta_h
\end{align}
where 
\begin{align}
	C_\omega&=\bmt
	1& 0&0 &\cdots &0 &-\omega\\ 
	-1 &  1 &0 & \cdots &0 &0\\
	0&-1 &  1 & 0 & \cdots &0\\
	0&0&\ddots &\ddots &\ddots &0\\
	0& 0&\cdots &-1& 1& 0\\
	0&0&\cdots & 0 & -1 &1
	\emt \in\IR^{N_t\times N_t}.  
\end{align}	
Let $\mathbb{F}=\frac{1}{\sqrt{N_t}}\left[\theta^{(l_1-1)(l_2-1)}\right]_{l_1, l_2=1}^{N_t}$ (with $\Ii=\sqrt{-1}$ and $\theta=e^{\frac{2\pi{\Ii}}{N_t}}$) 
be the discrete Fourier matrix.  For any given scalar number $\omega\ne 0$, define a diagonal matrix 
$	\Gamma_\omega=\diag \{ 1,\omega^{\frac{1}{N_t}},\cdots,\omega^{\frac{N_t-1}{N_t}}\} \in\IC^{N_t\times N_t}$. 
It is well-known \cite{BLM05} that the $\alpha$-circulant matrix $C_\omega$  admits a explicit diagonalization  $C_\omega=V D V^{-1}$,
where $V=\Gamma_\omega^{-1} \IF^*$, $V^{-1}=\IF\Gamma_\omega$, and 
$D=\mathrm{diag}\left(\sqrt{N_t}\mathbb{F}\Gamma_\omega C_\omega(:,1)\right)$  with $C_\omega(:,1)$ being the first column of $C_\omega$.
With this explicit diagonalization $C_\omega=V D V^{-1}$, we can factorize $	\widehat A_h$ into
the product form 
\[
\widehat A_h=\underbrace{(V\otimes I_x)}_{\tn{Step-(a)}}\underbrace{\left(\frac{1}{\tau} D\otimes I_x-I_t\otimes \Delta_h \right)}_{\tn{Step-(b)}} \underbrace{(V^{-1}\otimes I_x)}_{\tn{Step-(c)}}.
\]
Hence,  let $\widehat F=\texttt{mat}(\widehat f_h)\in\IR^{N_x\times N_t}$,  the solution vector $y_h=(\widehat A_h)^{-1} \widehat f_h$ can be computed via the following 3 steps:
\begin{equation}\label{3step}
	\begin{split}
		&\text{Step-(a)} ~~S_1=\widehat F (V^{-1})^{\T}  \in\IR^{N_x\times N_t},\\
		&\text{Step-(b)} ~~S_{2}(:,j)=\left( {{\tau}^{-1}d_{j}}I_x - \Delta_h\right)^{-1} S_{1}(:,j),\quad ~j=1,2,\dots,N_t,\\
		&\text{Step-(c)} ~~y_h=\texttt{vec}(S_2V^\T) \in\IR^{N_x N_t},\\
	\end{split}
\end{equation} 
where $D=\text{diag}(d_{1},\dots,d_{N_t})$ and $S_{1,2}(:,j)$ denotes the $j$-th column of $S_{1,2}$.
Here we have used the well-known Kronecker product property $(B \otimes I_x)\texttt{vec}(X)=\texttt{vec}( XB^\T)$ for any compatible matrices $B$ and $X$. 
Clearly, the $N_t$ complex-shifted linear systems of size $N_x\times N_x$ in Step-(b) can be computed in parallel since they are independent of each other. 
Notice that a different spatial operator would only affects the  matrix $\Delta_h$ in Step-(b).
Moreover, due to $V=\Gamma_\omega^{-1} \IF^*$, the matrix multiplication in Step-(a) and Step-(c)
can be computed efficiently via FFT  in time direction with $\CO(N_xN_t\log N_t)$ complexity.
In serial computation with sparse direct solver for Step-(b), the total complexity is
$\CO(2N_xN_t\log N_t+N_tN_x^3)$, which is significantly lower than the  $\CO(N_t^3N_x^3)$ complexity of sparse direct solver for the all-at-once system.
 Further parallel speedup can be  achieved in parallel computing, we refer to \cite{gander2020paradiag,caklovic2021parallel} for similar parallel  results. Finally, we highlight that the circulant structure of $C_\omega$ is not essential to such a PinT direct solver. In fact, any efficient diagonalization of $C_\omega=VDV^{-1}$ with a well-conditioned eigenvector matrix $V$ would be sufficient, but the computation of Step-(a) and (c) may need higher complexity unless $V$ has a special structure for fast computation. 
 \section{Convergence analysis and the choice of regularization parameter}
 In this section we present the error estimates for both PinT-QBVM and PinT-MQBVM under suitable assumptions.
 Let $\IA=-\Delta$ and define a Hilbert function space $H=H_0^1(\Omega)$ equipped with the $L^2$ norm $\|f\|_2:=\left(\int_{\Omega}f^2 dx\right)^{1/2}$.
 Then $\IA$ admits a set of orthonormal eigenbasis $\{\phi_l\}_{l\ge 1}$ in $H$, associated
 to a set of eigenvalues $\{\lambda_l\}_{l\ge 1}$ such that $\IA\phi_l=\lambda_l\phi_l$ with
$ 0<\lambda_1<\lambda_2<\cdots$ and $\lim_{l\to +\infty} \lambda_l=+\infty$. 
 Give any $g\in H$, it has a series expansion
$
 g=\sum_{l=1}^\infty b_l \phi_l,
$
 with $b_l=(g,\phi_l):=\int_\Omega g \phi_l dx$ for all $l\ge 1$.
 Let $S(t)=e^{-\IA t}$ be the compact contraction semi-group generated by $(-\IA)=\Delta$. 
 The noise-free un-regularized unique solution $z$ of the original BHCP (\ref{state}) has a formal series expression
 \begin{align} \label{exactsol}
 	z(\cdot,t)=S(t-T)g=S(t)(S(T))^{-1}g= \sum_{l=1}^\infty b_l e^{\IA (T-t)}\phi_l=\sum_{l=1}^\infty e^{(T-t)\lambda_l}b_l\phi_l,
 \end{align}
 which is numerically unstable for computing $z(\cdot,0)$, but very useful in convergence analysis.	
 For a given $g\in H$,
 it was proved in  \cite[Lemma 1]{Clark94} that the original problem (\ref{state})
 has a unique (classical) solution  if and only if 
 $
 \|z(\cdot,0)\|<\infty.
 $ 
 Hence we assume that $\|z(\cdot,0)\|_2^2=\sum_{l=1}^\infty  e^{2T\lambda_l} b_l^2 \le E_0^2<\infty$ for a constant $E_0>0$, which guarantees a unique solution. 
 
 Let $y_\alpha^Q,y_\alpha^M,y_\alpha^{PQ}$,and $y_\alpha^{PM}$ denotes the noise-free regularized solution of QBVM, MQBVM, PinT-QBVM, and PinT-MQBVM, respectively.
 Then we can easily verify the following explicit series representations 
 \begin{align}
 	y_\alpha^Q(\cdot,t)&=S(t)(\alpha I+S(T))^{-1}g=\sum_{l=1}^\infty \frac{e^{-t\lambda_l }}{\alpha+e^{-T\lambda_l }}b_l\phi_l,\label{regsolQ} \\
 		y_\alpha^M(\cdot,t)&=S(t)(\alpha \IA+S(T))^{-1}g=\sum_{l=1}^\infty \frac{e^{-t\lambda_l }}{\alpha\lambda_l+e^{-T\lambda_l }}b_l\phi_l,\label{regsolM}\\
 			y_\alpha^{PQ}(\cdot,t)&=S(t)(\alpha (I+\tau \IA)+S(T))^{-1}g=\sum_{l=1}^\infty \frac{e^{-t\lambda_l }}{\alpha(1+\tau\lambda_l)+e^{-T\lambda_l }}b_l\phi_l,\label{regsolPQ}\\
 				y_\alpha^{PM}(\cdot,t)&=S(t)(\alpha (\IA+I/\tau)+S(T))^{-1}g=\sum_{l=1}^\infty \frac{e^{-t\lambda_l }}{\alpha(\lambda_l+1/\tau)+e^{-T\lambda_l }}b_l\phi_l,\label{regsolPM}
\end{align} 
where $\tau>0$ is assumed to be fixed.
Clearly $y_\alpha^{PQ}$ with $\tau=0$ reduces to $y_\alpha^Q$
and $y_\alpha^{PM}$ with $\tau=\infty$ leads to $y_\alpha^M$.

We first discuss the stability estimates.
 Notice $1\le (1+\tau\lambda_l)\le e^{\tau\lambda_l}$ for any $\tau\lambda_l\ge 0$, for PinT-QBVM we have 
\begin{align} \label{stabilityPQ}
\|y_\alpha^{PQ}(\cdot,t)\|_2^2&=	
\sum_{l=1}^\infty  \frac{e^{-2t\lambda_l }}{(\alpha(1+\tau\lambda_l)+e^{-T\lambda_l })^2}  b_l^2 
=	
\sum_{l=1}^\infty  \frac{e^{-2t\lambda_l }}{(\alpha(1+\tau\lambda_l)+e^{\tau\lambda_l}e^{-(T+\tau)\lambda_l })^2}  b_l^2 \nonumber\\
&\le \sum_{l=1}^\infty  \frac{e^{-2t\lambda_l }}{(\alpha + e^{-(T+\tau)\lambda_l })^2(1+\tau\lambda_l)^2}  b_l^2 
\le 
\sum_{l=1}^\infty  \frac{e^{-2t\lambda_l }}{(\alpha +e^{-(T+\tau)\lambda_l })^{2-2t/(T+\tau)}
(\alpha +e^{-(T+\tau)\lambda_l })^{2t/(T+\tau)}}  b_l^2 \nonumber\\
&\le \sum_{l=1}^\infty  \frac{1}{(\alpha +e^{-(T+\tau)\lambda_l })^{2-2t/(T+\tau)}
 }  b_l^2 \le  {(1/\alpha)^{2(1-t/(T+\tau))}} \sum_{l=1}^\infty b_l^2=
 {(1/\alpha)^{2(1-t/(T+\tau))}}\|g\|^2_2,
\end{align} 
which, upon setting $\tau=0$, gives the known estimate \cite{Clark94} of QBVM:
$
	\|y_\alpha^{Q}(\cdot,t)\|_2^2\le {(1/\alpha)^{2(1-t/T)}}\|g\|^2_2.
$
Similarly, for PinT-MQBVM there holds
\begin{align} \label{stabilityPM}
	\|y_\alpha^{PM}(\cdot,t)\|_2^2&=	
	\sum_{l=1}^\infty  \frac{\tau^2 e^{-2t\lambda_l }}{(\alpha(1+\tau\lambda_l)+\tau e^{-T\lambda_l })^2}  b_l^2
	\le 
	\sum_{l=1}^\infty  \frac{\tau^2 e^{-2t\lambda_l }}{(\alpha+\tau e^{-(T+\tau)\lambda_l })^2}  b_l^2
	\nonumber\\ 
	&=
	\sum_{l=1}^\infty  \frac{\tau^2 e^{-2t\lambda_l }}{(\alpha +\tau e^{-(T+\tau)\lambda_l })^{2-2t/(T+\tau)}
		(\alpha +\tau e^{-(T+\tau)\lambda_l })^{2t/(T+\tau)}}  b_l^2 
 \le \sum_{l=1}^\infty  \frac{\tau^{2-2t/(T+\tau)}}{(\alpha +\tau e^{-(T+\tau)\lambda_l })^{2-2t/(T+\tau)}
	}  b_l^2 \nonumber\\
&\le {(\tau/\alpha)^{2(1-t/(T+\tau))}} \sum_{l=1}^\infty b_l^2=
	 {(\tau/\alpha)^{2(1-t/(T+\tau))}} \|g\|^2_2.
\end{align} 
Next, we estimate the regularization errors.
 Based on the series expansions of $z(t)$ and $y_\alpha^{PQ}(\cdot,t)$, we have 
\begin{align} \label{errorPQ}
\|y_\alpha^{PQ}(\cdot,t)- z(\cdot,t)\|^2&=
\sum_{l=1}^\infty \left(\frac{e^{-t\lambda_l }}{\alpha(1+\tau\lambda_l)+e^{-T\lambda_l }}-e^{(T-t)\lambda_l} \right)^2b_l^2  
=
\sum_{l=1}^\infty  \frac{\alpha^2(1+\tau\lambda_l)^2e^{2(T-t)\lambda_l} }{(\alpha(1+\tau\lambda_l)+e^{-T\lambda_l })^2}   b_l^2 \nonumber\\
& \le 
\sum_{l=1}^\infty  \frac{\alpha^2 e^{2(T-t)\lambda_l} }{\left(\alpha+e^{-(T+\tau)\lambda_l } \right)^2}   b_l^2
=
\sum_{l=1}^\infty  \frac{\alpha^2 e^{2T\lambda_l} }{(\alpha+e^{-(T+\tau)\lambda_l })^{2-2t/(T+\tau)}}\frac{ e^{-2t\lambda_l}}{(\alpha+e^{-(T+\tau)\lambda_l })^{2t/(T+\tau)}}   b_l^2 \nonumber\\
&\le 
\sum_{l=1}^\infty  \frac{\alpha^2  }{(\alpha+e^{-(T+\tau)\lambda_l })^{2-2t/(T+\tau)} } e^{2T\lambda_l} b_l^2 
\le \alpha^{2t/(T+\tau)} \sum_{l=1}^\infty e^{2T\lambda_l} b_l^2  
\le E_0^2\alpha^{2t/(T+\tau)}.
\end{align}  
Based on the series expansions of $z(t)$ and $y_\alpha^{PM}(\cdot,t)$, we have the error estimate  
\begin{align} \label{errorPM}
	\|y_\alpha^{PM}(\cdot,t)- z(\cdot,t)\|^2&=
	\sum_{l=1}^\infty \left(\frac{\tau e^{-t\lambda_l }}{\alpha(1+\tau\lambda_l)+\tau e^{-T\lambda_l }}-e^{(T-t)\lambda_l} \right)^2b_l^2  
	=
	\sum_{l=1}^\infty  \frac{\alpha^2  (1+\tau\lambda_l)^2e^{2(T-t)\lambda_l} }{(\alpha(1+\tau\lambda_l)+\tau e^{-T\lambda_l })^2}   b_l^2\nonumber \\
	& 
	 \le 
	\sum_{l=1}^\infty  \frac{\alpha^2 e^{2(T-t)\lambda_l} }{\left(\alpha+\tau e^{-(T+\tau)\lambda_l } \right)^2}   b_l^2 =
	\sum_{l=1}^\infty  \frac{\alpha^2 e^{2T\lambda_l} }{(\alpha+\tau e^{-(T+\tau)\lambda_l })^{2-2t/(T+\tau)}}\frac{ e^{-2t\lambda_l}}{(\alpha+\tau e^{-(T+\tau)\lambda_l })^{2t/(T+\tau)}}   b_l^2 \nonumber\\
	&\le 
	\sum_{l=1}^\infty  \frac{\alpha^2  }{(\alpha+\tau e^{-(T+\tau)\lambda_l })^{2-2t/(T+\tau)} } \frac{1}{\tau^{2t/(T+\tau)}} e^{2T\lambda_l} b_l^2 \nonumber\\
	&\le (\alpha/\tau)^{2t/(T+\tau)} \sum_{l=1}^\infty e^{2T\lambda_l} b_l^2 
	\le E_0^2(\alpha/\tau)^{2t/(T+\tau)},
\end{align} 
where the small $\tau>0$ in denominator will affect the choice of optimal parameter $\alpha$ for PinT-MQBVM.

The following theorem shows PinT-QBVM and PinT-MQBVM can achieve the same convergence rate.
Let $y_{\alpha,\delta}^{PQ}$ and $y_{\alpha,\delta}^{PM}$ denotes the noisy regularized solution of each method with $g$ be replaced by noisy $g^{\delta}$, respectively. 
%Recall $\|g\|_2^2=\sum_{l=1}^\infty  b_l^2< \|z(\cdot,0\|_2^2<\infty$.  
Assume the noisy measurement $g_{\delta}=\sum_{l=1}^\infty b_l^{\delta} \phi_l$ be in $H$ 
such that $\|g-g_{\delta}\|_2^2=\sum_{l=1}^\infty (b_l-b_l^{\delta})^2\le \delta^2$ holds for some  $\delta>0$. 
\begin{theorem} \label{ThmNoiseConv}
	Assume
	that $\|z(\cdot,0\|_2\le E_0$ for some constant $E_0>0$ and $\|g-g_{\delta}\|_2\le \delta$ for some $\delta>0$. 
	Then for PinT-QBVM with the choice $\alpha=\delta/E_0$ and PinT-MQBVM with the choice  $\alpha=\tau\delta/E_0$ there hold
	\begin{align} \label{deltaPQPM}
		\|y_{\alpha,\delta}^{PQ}(\cdot,t)-z(\cdot,t)\|_2 \le \sqrt{2} E_0^{(1-t/(T+\tau))}\delta^{t/(T+\tau)} \tn{and} 
		\|y_{\alpha,\delta}^{PM}(\cdot,t)-z(\cdot,t)\|_2 \le \sqrt{2} E_0^{(1-t/(T+\tau))}\delta^{t/(T+\tau)}.
	\end{align}  
\end{theorem}	
\begin{proof}
Combining the above estimates (\ref{stabilityPQ}) and (\ref{errorPQ}) for PinT-QBVM, there holds
\begin{align*}
\|y_{\alpha,\delta}^{PQ}(\cdot,t)-z(\cdot,t)\|_2^2
\le \|y_{\alpha,\delta}^{PQ}(\cdot,t)-y_{\alpha}^{PQ}(\cdot,t)\|_2^2+\|y_{\alpha}^{PQ}(\cdot,t)-z(\cdot,t)\|_2^2
\le {(1/\alpha)^{2(1-t/(T+\tau))}}\delta^2+E_0^2 \alpha^{2t/(T+\tau)},
\end{align*}	
which, upon choosing $\alpha=\delta/E_0$ such the two terms are equal, leads to the following error estimate
\begin{align} \label{deltaPQ}
\|y_{\alpha,\delta}^{PQ}(\cdot,t)-z(\cdot,t)\|_2 \le \sqrt{2} E_0^{(1-t/(T+\tau))}\delta^{t/(T+\tau)}.
\end{align} 
Similarly, with the estimates (\ref{stabilityPM}) and (\ref{errorPM}) for PinT-MQBVM we can obtain
\begin{align*}
	\|y_{\alpha,\delta}^{PM}(\cdot,t)-z(\cdot,t)\|_2^2
	\le \|y_{\alpha,\delta}^{PM}(\cdot,t)-y_{\alpha}^{PM}(\cdot,t)\|_2^2+\|y_{\alpha}^{PM}(\cdot,t)-z(\cdot,t)\|_2^2
	\le {(\tau/\alpha)^{2(1-t/(T+\tau))}} \delta^2+E_0^2(\alpha/\tau)^{2t/(T+\tau)},
\end{align*}	
which, upon choosing $\alpha=\tau\delta/E_0$ such the two terms are equal, gives exactly the  same error estimate as in (\ref{deltaPQ}):
\begin{align} \label{deltaPM}
\|y_{\alpha,\delta}^{PM}(\cdot,t)-z(\cdot,t)\|_2\le \sqrt{2}E_0^{1-t/(T+\tau)}\delta^{t/(T+\tau)}.
\end{align}	
This completes the proof. 
\end{proof}
Hence, PinT-QBVM with $\alpha=\delta/E_0$
and  PinT-MQBVM with $\alpha=\tau\delta/E_0$ have the same convergence rate,
which is also expected in discrete setting by observing the PinT-QBVM system (\ref{linsysQBVPPinT}) with $\alpha=\delta/E_0$ is identical to the PinT-MQBVM system (\ref{linsysQBVPmodifiedPinT}) with $\alpha=\tau\delta/E_0$.
Under the given assumptions, the best possible worst case error \cite{Tautenhahn_1996} is of the order $(E_0^{1-t/T}\delta^{t/T})$,
which implies the estimate (\ref{deltaPQPM}) is asymptotically optimal as $\tau\to 0$.
In particular, at the initial time $t=0$, the error estimate (\ref{deltaPQPM})  does not imply convergence since it only gives an upper bound $\sqrt{2}E_0$. This convergence rate may be further improved if imposing stronger assumptions on  $z(\cdot,0)$, we refer to \cite{Hao_2009,liu2019quasi} for related discussion.
 
%\begin{remark}

%
%In PinT-MQBVM,
% if choosing $\alpha=\sqrt{\tau}\delta/E_0$, there holds
% 	\begin{align} \label{deltaPM2}
% 		\|y_{\alpha,\delta}^{PM}(\cdot,t)-z(\cdot,t)\|_2^2&\le  {(\sqrt{\tau}E_0/\delta)^{2(1-t/(T+\tau))}} \delta^2+E_0^2(\delta/(E_0\sqrt{\tau}))^{2t/(T+\tau)}\nonumber\\
% 		&=E_0^{2(1-t/(T+\tau))} \delta^{2t/(T+\tau)} \left(\tau^{1-t/(T+\tau)}+\tau^{-t/(T+\tau)} \right),
% 	\end{align}	
% which leads to
% 	\begin{align} \label{deltaPM2T0}
% 	\|y_{\alpha,\delta}^{PM}(\cdot,0)-z(\cdot,0)\|_2^2&\le  E_0^2 \left(\tau+1 \right),
% \end{align}	
%\end{remark}
 \section{Numerical examples} \label{secNum}
 In this section, we present  numerical examples to illustrate the high computational efficiency of our proposed PinT algorithms. All simulations are  implemented in serial with MATLAB on a laptop PC with Intel(R) Core(TM) i7-7700HQ CPU@2.80GHz CPU and 48GB RAM,
 where CPU times (in seconds) are estimated by the timing functions \texttt{tic/toc}.
 For comparison, we accurately solve the full sparse linear systems and those independent complex-shifted linear system in Step-(b) of our PinT solver with
 MATLAB's highly optimized backslash sparse direct
 solver, which runs very fast for
 several thousands (but not millions) of unknowns.
 Parallel speedup results will be reported elsewhere.
 
 We generate the noisy final condition measurement by $g_{\delta}=g\times (1+\epsilon\times \rm{rand}(-1,1)),$
 where $\epsilon> 0$ controls the noise level
 and $\rm{rand}(-1,1)$ denotes random noise uniformly distributed within $[-1,1]$. We then further compute
 the estimated noise bound $\delta:=\|g^{\delta}-g\|_2$. Since $E_0$ is in general unknown, the optimal regularization parameter will be chosen as $\alpha=\delta$ for QBVM, MQBVM, and PinT-QBVM  and $\alpha=\tau\delta$ for PinT-MQBVM.
 Upon solving the discretized full linear system, we   obtain the approximate initial condition $y^0$ and then compute its discrete $L_2(\Omega)$ norm error as
 $
 e_h=\|y^0-z(\cdot,0)\|_2.
 $
 For a fixed mesh size, we would expect $e_h$ to decrease  as the noise level $\delta$ gets smaller, but the discretization errors also play a role.
 
 \subsection{Example 1 \cite{Hao_2009}: 1D with a non-smooth initial condition.}
 Choose $\Omega=(0,\pi), T=1$, and a non-smooth (triangular shaped) initial condition
 \[z(x,0) =\begin{cases} 
 	2x, & 0\le x\le \pi/2, \\
 	2(\pi-x), & \pi/2\le x\le \pi, \\ 
 \end{cases}
 \]
 which gives the exact solution (truncated the first 100 terms as benchmark reference in computation)
 \[
 z(x,t)=\frac{8}{\pi}\sum_{k=1,3,5,\cdots}^\infty \frac{ \cos(k(2x-\pi)/2)}{k^2 } e^{-k^2 t}.
 \]
Since the initial data $z(x, 0)$ is not in $C^1(0,\pi)$, we would not expect to be able to approximate it very accurately.  
Figure \ref{FigEx2} plots the reconstructed initial conditions by different regularization methods against the true initial data $z(x,0)$. 
In  Table \ref{T2A}, we compare the errors and CPU times of different methods for a sequence of decreasing mesh sizes with different levels of noise.
The CPU times indicate both PinT-QBVM and PinT-MQBVM are about \textbf{50 times} faster than QBVM and MQBVM, which is anticipated based on our discussion. 
The reconstruction of PinT-QBVM is more accurate than QBVM with spurious oscillations, while PinT-MQBVM  delivers a comparable accuracy as PinT-QBVM.
The reconstructed initial data by PinT-QBVM and PinT-MQBVM indeed show some irregular fluctuation, which can be smoothed out by choosing a larger (possible non-optimal) regularization parameter. 
Figure \ref{FigEx2B} illustrates the reconstructed initial conditions with $\alpha=\delta/\sqrt{\tau}$ for QBVM, MQBVM, and PinT-QBVM  and $\alpha=\sqrt{\tau}\delta$ for PinT-MQBVM, where the obtained approximations are more regular (smooth) but less accurate (especially for large $\delta$).
One possible explanation for such difference is due to the different condition numbers that magnify the noise.
  \begin{figure}[htp!]
 	\begin{center}
 		\includegraphics[width=1\textwidth]{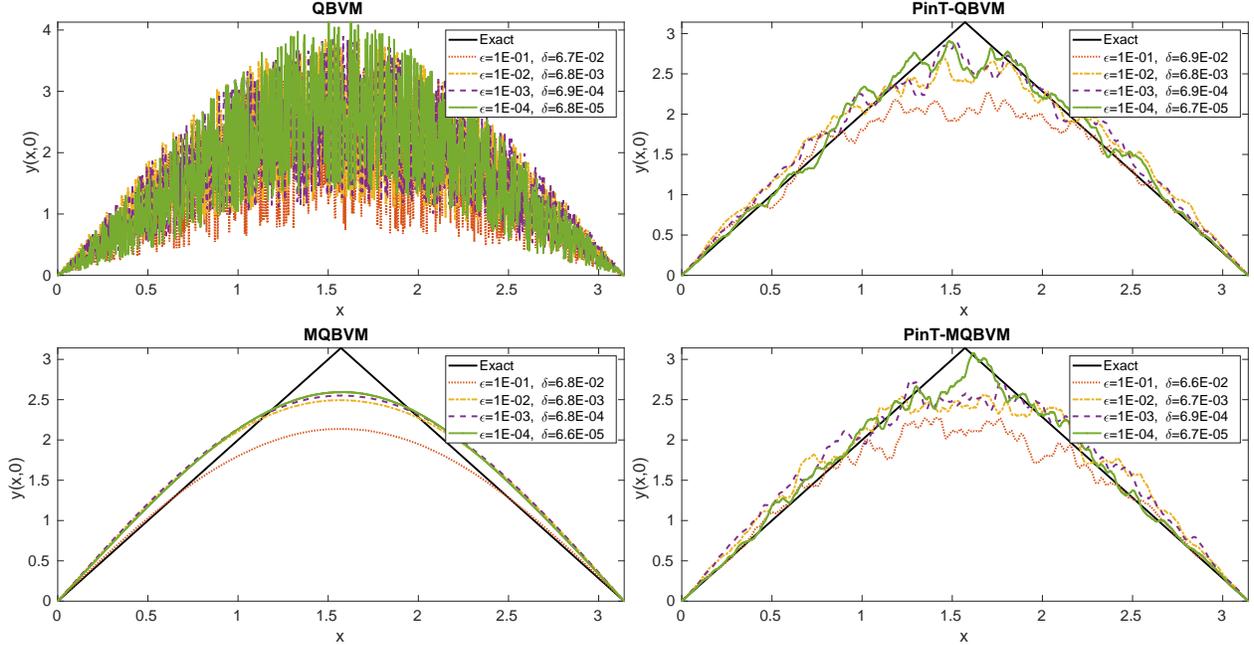}
 	\end{center}
 	\vspace{-2em}
 	\caption{Reconstructed $y(x,0)$ in Ex. 1 with different methods and noise levels $\epsilon\in\{10^{-1},10^{-2},10^{-3},10^{-4}\}$ (using the mesh $h=\pi/1024,\tau=T/1024$, $\alpha=\delta$ for QBVM, MQBVM, and PinT-QBVM  and $\alpha=\tau\delta$ for PinT-MQBVM). The black solid curve is the exact solution.} 
 	
 	\label{FigEx2}
 \end{figure} 
 
 \begin{table}[htp!] 
 	\centering
 	\caption{{Error and CPU results for Ex. 1 with different mesh sizes and noise levels.}}
 	\begin{tabular}{|c|c||cccc||cccc|}\hline 
 		& &\multicolumn{4}{|c|}{Errors in $L_2$ norm}&\multicolumn{4}{|c|}{CPU (in seconds)} \\
 		\hline
 		Method	&$(N_x,N_t)$$\backslash$   $\epsilon$&	 $10^{-1}$&	$10^{-2}$&$10^{-3}$&$10^{-4}$  &	 $10^{-1}$&	$10^{-2}$&$10^{-3}$&$10^{-4}$      \\ \hline

 		\multirow{3}{*}{\parbox{2cm}{QBVM\\($\alpha=\delta$)}}
&(256, 256)	 &1.177 	 &1.076 	 &1.070 	 &1.021 	&0.41 	 &0.44 	 &0.42 	 &0.45 	\\
&(512, 512)	 &1.167 	 &1.075 	 &1.063 	 &1.033 	&1.82 	 &1.75 	 &1.71 	 &1.83 	\\
&(1024,1024)	 &1.165 	 &1.061 	 &1.055 	 &1.024 	&10.71 	 &10.89 	 &10.67 	 &10.80 	\\
 		\hline \hline
 	\multirow{3}{*}{\parbox{2cm}{PinT-QBVM\\($\alpha=\delta$)}}
&(256, 256)	 &0.674 	 &0.459 	 &0.354 	 &0.344 	&0.02 	 &0.02 	 &0.02 	 &0.02 	\\
&(512, 512)	 &0.648 	 &0.427 	 &0.398 	 &0.217 	&0.06 	 &0.05 	 &0.06 	 &0.06 	\\
&(1024,1024)	 &0.683 	 &0.400 	 &0.329 	 &0.279 	&0.22 	 &0.21 	 &0.23 	 &0.21 	\\
 		\hline \hline 
 	\multirow{3}{*}{\parbox{2cm}{MQBVM\\($\alpha=\delta$)}}
&(256, 256)	 &0.640 	 &0.382 	 &0.374 	 &0.337 	&0.44 	 &0.43 	 &0.43 	 &0.43 	\\
&(512, 512)	 &0.640 	 &0.394 	 &0.388 	 &0.329 	&1.75 	 &1.75 	 &1.70 	 &1.84 	\\
&(1024,1024)	 &0.641 	 &0.394 	 &0.375 	 &0.331 	&10.80 	 &10.73 	 &11.09 	 &10.90 	\\
 \hline \hline 
 		\multirow{3}{*}{\parbox{2.2cm}{PinT-MQBVM\\($\alpha=\tau\delta$)}}
&(256, 256)	 &0.654 	 &0.439 	 &0.420 	 &0.292 	&0.02 	 &0.02 	 &0.02 	 &0.02 	\\
&(512, 512)	 &0.667 	 &0.439 	 &0.345 	 &0.275 	&0.05 	 &0.06 	 &0.06 	 &0.06 	\\
&(1024,1024)	 &0.638 	 &0.442 	 &0.397 	 &0.214 	&0.21 	 &0.22 	 &0.22 	 &0.22 	\\
 		\hline
 	\end{tabular}
 	\label{T2A}
 \end{table} 
 
 \begin{figure}[htp!]
 	\begin{center}
 		\includegraphics[width=1\textwidth]{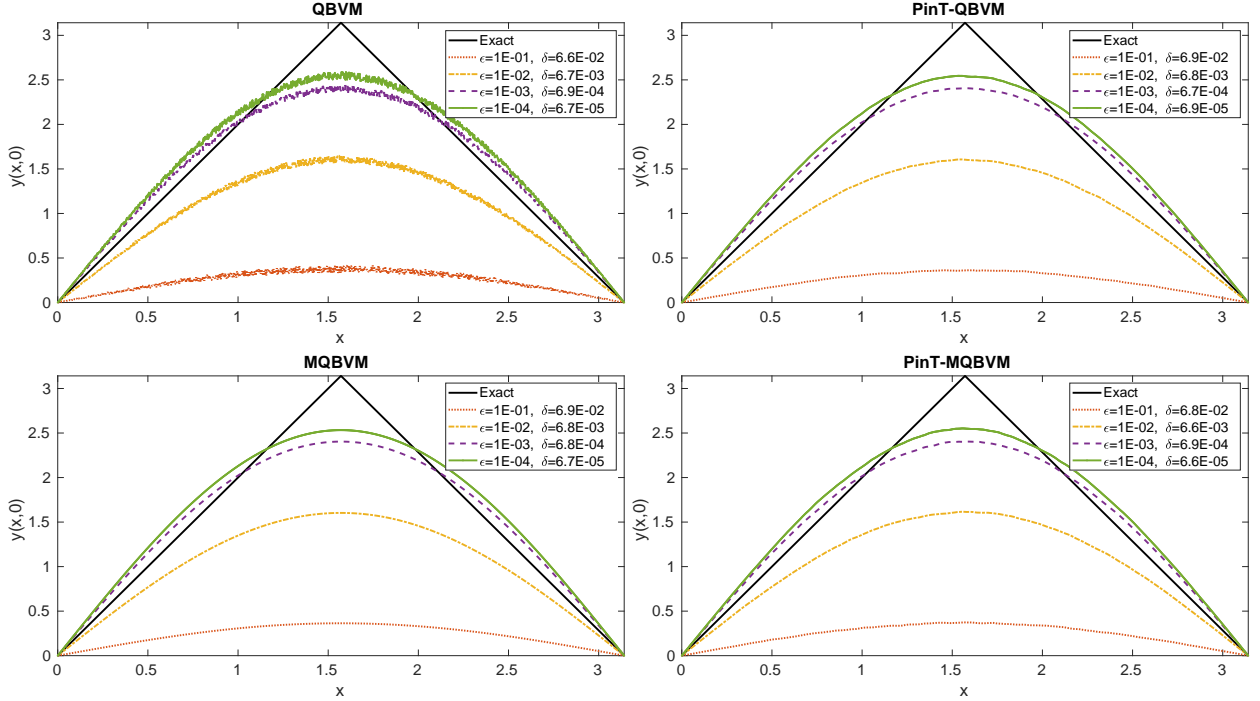}
 	\end{center}
 	\vspace{-2em}
 	\caption{Reconstructed $y(x,0)$ in Ex. 1 with different methods and noise levels $\epsilon\in\{10^{-1},10^{-2},10^{-3},10^{-4}\}$ (using the mesh $h=\pi/1024,\tau=T/1024$, $\alpha=\delta/\sqrt{\tau}$ for QBVM, MQBVM, and PinT-QBVM  and $\alpha=\sqrt{\tau}\delta$ for PinT-MQBVM). The black solid curve is the exact solution.} 
 	
 	\label{FigEx2B}
 \end{figure}

 \subsection{Example 2 \cite{liu2019quasi}: 2D with a smooth initial condition.}
 Choose $\Omega=(0,\pi)^2, T=1$ and 
 $z(x_1,x_2,T)=e^{-2T}\sin( x_1)\sin( x_2)$, which gives the exact solution
 $$
 z(x_1,x_2,t)=e^{-2t}\sin( x_1)\sin( x_2).
 $$
 Figure \ref{FigEx3} plots the reconstructed initial conditions by different regularization methods with different levels of noise
 and Table \ref{T3A} compared the errors and CPU times.
 Except for QBVM, the approximation errors (including both noise and discretization errors) show a comparable convergence as $\epsilon$ is decreased and the mesh is refined.
 Here ``--'' denotes the MATLAB's backslash solver fails to solve the systems due to excessively long computation time.
 For 2D problems, even with a medium-scale mesh size $(N_x,N_t)=(128^2,128)$, 
  the discretized linear systems (with about 2 million unknowns) by QBVM and MQBVM already become too large to be solved by sparse direct solver within a reasonable time on the used PC, while our proposed PinT-QBVM and PinT-MQBVM take only less than 10 seconds with the designed PinT direct solver. This example demonstrates the high efficiency of our fast PinT direct solver in solving large-scale problems.
  To gain even better efficiency in more challenging 3D problems, fast and robust  iterative solvers, such as the multigrid method or the preconditioned Krylov subspace method, need to be developed, where our proposed PinT direct solver can be used as an effective preconditioner.
  It is desirable to design a fast iterative solver with mesh-independent and regularization-robust convergence rate, which is left as future work.
% \begin{figure}[htp!]
% 	\begin{center}
% 		\includegraphics[width=0.4\textwidth]{Ex3_exact_IC}
% 	\end{center}
% 	\vspace{-2em}
% 	\caption{The exact initial condition $y(x_1,x_2,0)=\sin( x_1)\sin( x_2)$ in Ex. 3  for comparison.} 
% 	
% 	\label{FigEx3exact}
% \end{figure} 

\begin{figure}[htp!]
	\begin{center}
		\includegraphics[width=1\textwidth]{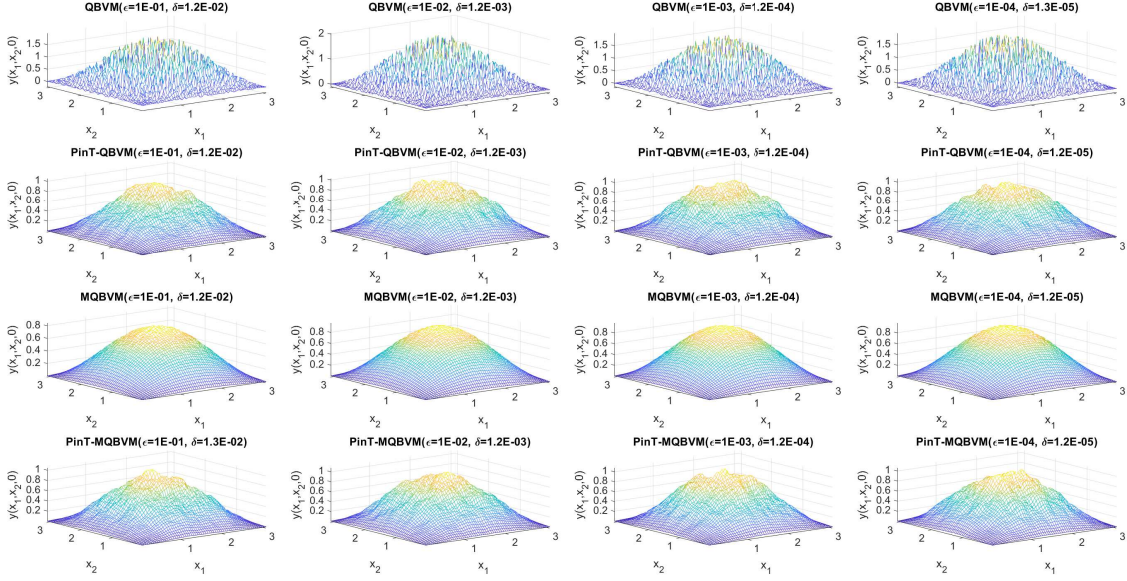}
	\end{center}
	\vspace{-2em}
	\caption{Reconstructed $y(x_1,x_2,0)$ in Ex. 2 with different methods (from top to bottom: QBVM, PinT-QBVM, MQBVM, PinT-MQBVM) and noise levels (from left to right) $\epsilon\in\{10^{-1},10^{-2},10^{-3},10^{-4}\}$ (using the mesh $h=\pi/64,\tau=T/64$).  } 
	
	\label{FigEx3}
\end{figure}

\begin{table}[htp!] 
	\centering
	\caption{{Error and CPU results for Ex. 2 with different mesh sizes and noise levels.}}
	\begin{tabular}{|c|c||cccc||cccc|}\hline 
		&&\multicolumn{4}{|c|}{Errors in $L_2$ norm}&\multicolumn{4}{|c|}{CPU (in seconds)} \\
		\hline
		Method	&$(N_x,N_t)$$\backslash$   $\epsilon$&	 $10^{-1}$&	$10^{-2}$&$10^{-3}$ &$10^{-4}$&	 $10^{-1}$&	$10^{-2}$&$10^{-3}$ &$10^{-4}$     \\ \hline

			\multirow{4}{*}{\parbox{2cm}{QBVM\\($\alpha=\delta$)}}
&($ 16^2$,  16)	 &1.032 	 &0.998 	 &1.001 	 &0.982 	&0.14 	 &0.04 	 &0.08 	 &0.11 	\\
&($ 32^2$,  32)	 &1.019 	 &1.004 	 &0.999 	 &1.002 	&1.06 	 &0.99 	 &1.15 	 &6.92 	\\
&($ 64^2$,  64)	 &1.013 	 &1.000 	 &1.001 	 &0.999 	&57.20 	 &57.12 	 &58.74 	 &59.04 	\\
&($128^2$, 128) & --&--&--&--&--&--&--&--\\ 
		\hline \hline
		\multirow{4}{*}{\parbox{2cm}{PinT-QBVM\\($\alpha=\delta$)}}
&($ 16^2$,  16)	 &0.356 	 &0.264 	 &0.259 	 &0.247 	&0.83 	 &0.02 	 &0.01 	 &0.01 	\\
&($ 32^2$,  32)	 &0.262 	 &0.177 	 &0.164 	 &0.170 	&0.15 	 &0.09 	 &0.09 	 &0.09 	\\
&($ 64^2$,  64)	 &0.212 	 &0.134 	 &0.103 	 &0.118 	&0.93 	 &0.88 	 &0.87 	 &0.88 	\\
&($128^2$, 128)	 &0.170 	 &0.087 	 &0.079 	 &0.080 	&7.69 	 &7.71 	 &7.68 	 &7.73 	\\
		\hline \hline 
	\multirow{4}{*}{\parbox{2cm}{MQBVM\\($\alpha=\delta$)}}
&($ 16^2$,  16)	 &0.370 	 &0.212 	 &0.197 	 &0.195 	&0.22 	 &0.04 	 &0.03 	 &0.03 	\\
&($ 32^2$,  32)	 &0.295 	 &0.122 	 &0.103 	 &0.101 	&1.04 	 &1.10 	 &1.05 	 &1.17 	\\
&($ 64^2$,  64)	 &0.270 	 &0.076 	 &0.054 	 &0.051 	&60.90 	 &61.00 	 &61.07 	 &61.39 	\\
&($128^2$, 128) & --&--&--&--&--&--&--&--\\ 
		\hline \hline 
		\multirow{4}{*}{\parbox{2.2cm}{PinT-MQBVM\\($\alpha=\tau\delta$)}}
&($ 16^2$,  16)	 &0.383 	 &0.271 	 &0.243 	 &0.257 	 &0.01 	 &0.01 	 &0.01 	 &0.01 	\\
&($ 32^2$,  32)	 &0.261 	 &0.185 	 &0.175 	 &0.170 	 &0.09 	 &0.09 	 &0.09 	 &0.09 	\\
&($ 64^2$,  64)	 &0.205 	 &0.113 	 &0.127 	 &0.107 	 &0.89 	 &0.89 	 &0.88 	 &0.90 	\\
&($128^2$, 128)	 &0.170 	 &0.090 	 &0.088 	 &0.084 	 &7.84 	 &7.88 	 &7.85 	 &7.87 	\\
		\hline
	\end{tabular}
	\label{T3A}
\end{table} 

\section{Conclusions}  \label{secFinal}
 Backward heat conduction problems are severely ill-posed and their stable numerical computation requires suitable regularization techniques.
 The quasi-boundary value method and its variants are widely used for regularizing such problems, 
 which upon space-time finite difference discretization leads to large-scale ill-conditioned nonsymmetric sparse linear systems.
 Such all-at-once linear systems are costly to solve by either direct or iterative methods.
 In this paper we have redesigned the well-established quasi-boundary value methods such that the full discretized system matrix admits a block $\omega$-circulant structure that can be solved by a fast diagonalization-based PinT direct solver. Convergence analysis (with the optimal  choice of regularization parameter $\alpha$) for our proposed PinT methods are given.
 Both 1D and 2D examples show our proposed PinT methods  can achieve a comparable accuracy with significantly faster CPU times (even in serial implementation). 
 The novel idea is to maneuver the flexibility of regularization for better structured systems that admit faster system solvers, which is applicable to a wide range of inverse PDE problems as well as other modern regularization methods.  
%\section*{Acknowledgments}
%The authors would like to thank the editor and two anonymous referees for their valuable comments and detailed revision
%suggestions that have greatly improved the overall quality of this paper.

{\small
\bibliographystyle{siam}
\bibliography{DirectPinT,inversePDE,waveControl}
}

\end{document}